\documentclass[a4paper,10pt]{amsart}
\usepackage{amsmath}
\usepackage{amssymb}
\usepackage{amstext}
\usepackage{amsfonts}
\usepackage{amssymb}

\newtheorem{theorem}{Theorem}[section]
\newtheorem{lemma}[theorem]{Lemma}
\newtheorem{corollary}[theorem]{Corollary}
\newtheorem{definition}[theorem]{Definition}
\newtheorem{proposition}[theorem]{Proposition}
\newtheorem{remark}[theorem]{Remark}

\def \R {{\mathbb {R}}}

\usepackage[usenames, dvipsnames]{color}
\begin{document}
\title[Moser Iteration ]{H\"older regularity for weak solutions of H\"ormander type operators}

\author[G.~Citti]{Giovanna Citti}
\address{Dipartimento di Matematica, Piazza di Porta S. Donato 5, 401
26 Bologna, Italy}
\email{giovanna.citti@unibo.it}

\author[M.~Manfredini]{Maria Manfredini}
\address{Dipartimento di Scienze Fisiche,Informatiche e Matematiche, Via Campi 213, 41125 Modena, Italy}
\email{maria.manfredini@unimore.it}

\author[Y.~Sire]{Yannick Sire}
\address{Department of Mathematics, Johns Hopkins University, Baltimore, MD 21218, USA}
\email{ysire1@jhu.edu}

\bigskip

\begin{abstract}
Motivated by recent results on the (possibly conditional) regularity for time-dependent hypoelliptic equations, we prove a parabolic version of the Poincar\'e inequality, and as a consequence, we deduce a version of the classical Moser iteration technique using in a crucial way the geometry of the equation. The point of this contribution is to emphasize that one can use the {\sl elliptic} version of the Moser argument at the price of the lack of uniformity, even in the {\sl parabolic } setting. This is nevertheless enough to deduce H\"older regularity of weak solutions. The proof is elementary and unifies in a natural way several results in the literature on Kolmogorov equations, subelliptic ones and some of their variations. 
\end{abstract}

\maketitle

\normalsize

\tableofcontents

\section{Introduction}

In this paper we prove H\"older regularity for solutions of a general Kolmogorov type equation with measurable coefficients, 
applying a version of the Moser iteration, usually used only in the elliptic setting. 

We consider a family of smooth linearly independent vector fields
$(X_i)_{i=0, \ldots, m}$,   spanning an homogeneous group of step $s$ (see Definition \ref{Vecf} below). Define the operator $\mathcal L$ by
\begin{equation}\label{op}
\mathcal L u
= -X_0u  + \sum_{i,j=1}^m X_i (a_{i,j}(x) X_j u),\    
\end{equation}
where $A:=(a_{i,j})_{i,j}$
is a symmetric uniformly elliptic matrix with {\sl measurable} coefficients. The previous operator includes the following cases: 

\begin{itemize}
\item The subelliptic operators associated to a suitable  group structure generated by the vector fields. 
\item The ultra-parabolic setting whenever $X_0$ is seen as a time derivative. This includes in particular the so-called Kolmogorov operators and Kolmogorov-Fokker-Planck considered previously in the literature. 
\end{itemize}

In the present paper, we investigate regularity via functional inequalities, i.e. using the Moser technique, for the following very simple equation 
\begin{equation}\label{eq}
\mathcal L u =   f  
,\;\text{  on an open set  } \Omega \subset \mathbb R^n
\end{equation}
whenever $f$ belongs to a suitable $L^p$ space.

The classical Moser’s iteration method (see \cite{moser1}) has been been largely used for elliptic and subelliptic operators. 
The  iteration technique for parabolic operators is different from the elliptic one, already in 
the Euclidean setting. The H\"older regularity of weak solutions follows classically as a consequence of an Harnack inequality 
on two different cylinder shifted in time (see \cite{moser2}). Recently in \cite{PP},  \cite{CPP},  \cite{APR} the authors proved for \eqref{eq}
the boundedness of the solution, in the parabolic and Kolmogorov setting via the Moser technique.  We would like also to mention  the recent paper \cite {DH} where the authors  prove a weak Harnack inequality for solutions of rough hypoelliptic equations under rather complicated assumptions.

In order to conclude to the holderianity of the solution in the elliptic case, it is necessary to apply a Poincar\'e inequality, 
which was missing in the parabolic case. Indeed in \cite{WZ}, the authors proved a very weak and technical version of the Poincar\'e inequality for the particular class
of Kolmogorov type equations introduced in \cite{LPo}, and consequently showed $C^\alpha$ regularity only for this class of operators. 
More recently a Poincar\'e inequality which holds only on the solutions has been proved in \cite{AR19}. 
Another approach based on fractional powers of the time derivatives was proposed in \cite{NS} for parabolic operators (see also \cite{GT2} for similar ideas). In this case the authors used separate Poincaré inequalities in the different variables. On the other hand properties of the fundamental solution and mean value formulas (see \cite{LP}) are known for very general 
Kolmogorov type operators, and they have largely been used to prove the Sobolev inequality. 

\smallskip

Our goal is to give an elementary proof of the standard Poincar\'e inequality in a very general context, which contains as a particular case the operator studied in \cite{WZ} and many others.  As a consequence Moser-type estimates, we establish the desired H\"older continuity for all weak solutions of \eqref{eq}. We believe that the main contribution of our work is the unifying perspective it provides, since we use elliptic techniques in a parabolic setting.  In this way we provide an elementary proof which is new even in the parabolic case. 

\smallskip

The definition of weak solutions reads as follows. 
\begin{definition}\label{weakSol}

A function $u$ is a weak solution of \eqref{op} in $W^{1, 2}_{\nabla} $ if it satisfies 
\begin{equation}\label{weak}
 \int u X_0 \phi \, dx- \int  \sum_{i,j=1}^m a_{i,j}(x) X_j u X_i \phi \, dx=\int \mathcal Lu\phi \, dx \quad \forall \phi \in C^\infty_0( \mathbb R^{n}), 
\end{equation}
where we denoted $\nabla u =( X_1, \ldots, X_m)$ and 
\begin{equation}\label{Sob}
W^{1, 2}_{\nabla} = \{u\in L^2: \nabla u \in L^2\}. 
\end{equation}

\end{definition}

The following result is a {\sl non uniform} Harnack inequality; more precisely only a {\sl locally uniform} Harnack inequality. The statement catches all the features of our proof, which is based on the geometry of the equation. 

\begin{theorem}\label{harnack}
Let $U_0 \subset \mathbb R^n$ be an open set for $n \geq 2$ and consider a weak solution (in the sense of Definition \ref{weakSol}) of the equation $\mathcal Lu=f \in L^p$ for some $p >Q/2 $  and where we denote $Q$ the homogeneous dimension associated to the family of vector fields $(X_0,X_1,...,X_m)$ (see Definition \ref{Vecf}). Let $U$ be such that 
\begin{itemize}
\item  $U\subset\subset U_0$ and 
\item for any given $x \in U$ there exists a family of open sets $\Omega(y,R)$ such that $\cup_{y \in B(x, 2R)}  \Omega(y,R) \subset \subset  U_0 $.
\end{itemize}
Then here exists a constant $C$, depending on $\sup_{U_{0}} u$ and $||f||_{L^p(U_0)}$  such that 
$$ \sup_{U}   u \leq C \inf_{U}  u. $$
\end{theorem}

The previous statement captures the geometry of the operator $\mathcal L$ via the homogeneous dimension $Q$, the Carnot-Caratheodory balls $B(x,R)$ and the family $\Omega(x,R)$ which will be taken to be the  $R$-sublevel set of the fundamental solution for a constant coefficient operator built out of $\mathcal L$. As in \cite{GT} (page 200), it leads classically 

\begin{corollary} 
Under the assumptions of Theorem \ref{harnack}, any weak solution of \eqref{eq} is locally H\"older continuous.  
\end{corollary}

It is important to remark that we prove a locally uniform Harnack inequality:  we fix an open set $U_0$ and 
for every open set $U\subset\subset U_0$ we obtain an Harnack inequality with constants 
depending on $||u||_{L^2(U_0)}$ (or as stated $\sup_{U_0} u$). This does not contradict the counter example of Moser for parabolic equations. 
Indeed if $u$ is the fundamental solution of the heat equation, there is no constant $C$ such that 
$ \sup u \leq C \inf  u $ on the exterior of a compact set; however, if we fix an open set $U_0$ and we 
allow the constant to depend on the $||u||_{L^2(U_0)}$, the Harnack inequality is satisfied.

Theorem \ref{harnack} holds on sets of the form
$$V_{R}=\cup_{y \in V}  \Omega(y,R) \subset \subset  U_0$$
for some sets $V$. In the classical elliptic case, the set $V$ can be chosen as ball and as well as all the sets $ \Omega(y,R)$. In the classical parabolic setting, if $V$ is a level set of the fundamental solution of the heat operator, then $V_R$ will not be a level set. To be able to match the upper and lower bounds on $u$, the usual choice for $V$ is a parabolic cylinder and in this case  $V_{R}$ is a cylinder, of double radius. In this way the structure of the sets is preserved leading to a scale invariant Harnack inequality. Here we loose the invariance property of the inequality; this is why we allow $V$ and $V_R$ to be completely different sets.

The Moser iteration (as well as the De Giorgi method) requires the use of Poincar\'e and Sobolev inequalities. By reducing to a constant coefficient operator $\tilde{\mathcal L}$ associated to $\mathcal L$, one can obtain a Sobolev inequality out of its fundamental solution. This leads to a standard  local Poincar\'e type inequality for general vector fields directly from the Sobolev inequality, with an extremely simple and direct proof. 

We define
\begin{equation}\label{op_constant}
\tilde{\mathcal L}u
=- X_{0}u + \sum_{i=1}^m  X^2_{i} u .
\end{equation}
Such an operator has always a fundamental solution since it has constant coefficients.

\begin{lemma}\label{poincareIntro}
 Let $U \subset \R^n$ be a bounded connected, open set  with $C^1$ boundary, and let $V\subset U$. The following hold:
 \begin{enumerate}
 \item There exists a constant $C$ depending only on $n, p, U,V$ such that  
\begin{equation}\label{poinc1Intro}
|| u - u_V||_{L^1(U)}\leq C diam(U)  (diam(U) ||\mathcal Lu||_{L^1(U)} +||\nabla u||_{L^1(U)}) 
\end{equation}
\item If the vector field $X_0$ commutes with all the vector fields $X_i$, we have
\begin{equation}\label{poinc2Intro}
|| u - u_V||_{L^2(U)}\leq C diam(U) ( ||X_0^{1/2} u||_{L^2(U)} + ||\nabla u||_{L^2(U)}) ,
\end{equation}
\end{enumerate}

In the previous inequalities $\nabla$ stands for the horizontal gradient introduced before (see Definition \ref{weakSol}). 
\end{lemma}

The inequality \eqref{poinc2Intro} was already introduced in \cite{NS}. It is important to point out that our method uses Moser's approach. Aiming at unifying both the elliptic and parabolic arguments in the Moser iterations, we actually fully use the geometry of the equation induced by the vector fields $(X_0,X_1,...,X_m)$. Several operators are included in our framework. As far as H\"older regularity is concerned, we point out also many recent results for kinetic equations (as Landau, Kolmogorov, Kolmogorov-Fokker-Planck) where the authors aim at proving H\"older regularity using the De Giorgi approach. The weak Harnack inequality has been for instance in the recent paper by Guerand and Mouhot \cite{GM};  De Giorgi's method has been also implemented in e.g. \cite{guerandImbert, golse}. We refer the reader to these papers and references therein and also the survey by Mouhot \cite{mouhotICM}.

\section{Preliminaries and known results}
. 

\subsection{Properties of vector fields }
We define 

\begin{definition}\label{Vecf}
We say that $(X_0, \ldots, X_m)$ generate a stratified structure of type $2$ and step $s$ if 
the tangent space admits the following decomposition 
\begin{equation}\label{tangentgenerate}T \R^n_{|p} = {V_1}_{|p} \oplus {V_2}_{|p}\oplus {V_s}_{|p}\quad \forall p \in \mathbb R^n
\end{equation}
where 
$$V_1 = span (X_1, \cdots, X_m), \quad V_2 = span (X_0, [V_1, V_1]), $$
$$ V_{j} = [V_{j-1}, V_1] \text{ for }j=3, \ldots, s-1, \quad  [V_1,V_s]=0.$$
\end{definition}

The stratification induces a natural notion of 
degree of a vector field:  
\begin{equation}\label{degree}\text{deg}(X)=j \quad\text {whenever}\; X\in V_j.\end{equation}

We will choose a stratified basis of the tangent space as follows. We will  complete the family $X_0, X_1, \cdots X_{m}$ to a basis of $V_2$, then complete it to a basis of $V_3$ and so on. Since the exponential mapping is a global isomorphism, we can assume that the coordinates of the space are the exponential ones, defined (see \cite{RS}) as 
\begin{equation}\label{chv} 
x = exp(\sum_{i=1}^{n} a_i X_{i-1} )\end{equation}
We  note that in the parabolic and Kolmogorov settings, $X_0 $ and $V_1$ are linearly independent; 
while if they are linearly dependent, then the operator is subelliptic. 
It is possible to find a change of variable in which $X_0$ is represented as the operator $\partial_t$, but the coefficients $X_i$ could potentially depend on $t$. However they will be independent of $t$ only in case that 
$X_0$  commutes with $X_1, \cdots, X_m$, which is not true in general. 

\begin{remark}\label{tnotation}
If 
$X_0$  commutes with $X_1, \cdots, X_m$, we will be in the parabolic case.  With the change of variable \eqref{chv} 
the vector $X_0$ will reduce to a partial derivative $\partial_{1}$. We will sometimes denote it $\partial_t$, and correspondingly we will call $t=x_1$, $\hat x= (x_2, \cdots, x_n)$  and the point of the space will be denoted $x=(t, \hat x)$ to clarify the parabolic nature of the problem. 
\end{remark}

In addition the conditions imposed on the vector fields ensure 
that via the exponential map, 
$\R^{n}$ is endowed with a homogeneous Lie  group structure 
and the resulting group is denoted by $\mathbb{G}$ (see for example \cite{BLU})
$$x \circ y = \exp \Big( \Phi(x) * \Phi(y)\Big) \, \forall x, y\in \mathbb{G}.$$
    Due to the stratification of the algebra, we can define a natural family of dilation  $(\delta_\lambda)_{\lambda>0}$ as follows:
for any $\delta > 0$, we define dilations on the vector fields generated by $(X_i)_{i=0}^m$ as 
\begin{equation}\label{dilation}
 \quad  \delta_\lambda (X_i) = \lambda X_i, \quad \delta_\lambda (X_0) = \lambda^2 X_0,\end{equation}
Due to the stratification, the dilation is a Lie algebra automorphism, namely,  
$$\delta_\lambda ([X; Y ]) = [\delta_\lambda X, \delta_\lambda Y]  \quad \forall
X, Y.$$
As a consequence, for every $\lambda$ the exponential map induces an automorphism 
on the group.

The braket generating condition we imposed in \eqref{tangentgenerate}, ensures that the vector fields $(X_i)_{i=0}^m$
satisfy the H\"ormander condition. Hence there is a Carnot-Carathéodory distance associated to these vector fields and denoted $d$. 
The ball $B(x, r)$  of the metric can be estimated in terms of the homogeneous dimension of the space, defined as 
\begin{equation}\label{homodim}
Q:= \sum_{i=1}^m i \, dim(V_i). \end{equation}
Indeed 
there exist constants $C_1, C_2$ such that 
$$
C_1 r^Q \leq |B(x,r)|\leq C_2 r^Q \qquad\forall\, r>0, \ x\in \mathbb G,$$
where $|\,\cdot\,|$ denotes the Lebesgue measure.

The associated classes of H\"older continuous functions 
will be defined as follows:
\begin{definition}\label{defholder}
Let $0<\alpha < 1$, $V\subset \mathbb{G}$ be an open set, and  $u$ be a function defined on
$V.$ We say that $u \in C^{\alpha}(V)$ if there exists a positive constant $M$ such that for
every $x, x_0\in V$ 
\begin{equation}\label{e301}
   |u(x) - u(x_0)| \le M   d ^\alpha(x, x_{0}).
\end{equation}
We denote 
 $$\|u\|_{C^{\alpha}(V)}=\sup_{x\neq x_{0}} \frac {|u(x) - u(x_{0})|}{d^\alpha(x, x_{0})}+ 
\sup_{V} |u|.$$
\end{definition}

\subsection{Derivatives and Fractional derivatives}

If $\phi$ is a continuous function defined in an open set $V$ of $\mathbb{G}$ and if, for every $i=1,\ldots, m,$ there exists the Lie derivative 
$X_{i}\phi$ then we  call horizontal gradient of $\phi$ the vector
\begin{equation}\label{e:tutto subriemannian}
	\nabla \phi=\sum_{i=1}^m (X_{i}\phi) X_{i}.
\end{equation} 
We explicitly note that the vector $X_0$ does not show up here, since it represents a vector field of degree 2 (with respect to dilations). 
In \eqref{Sob} we introduced the Sobolev space $W^{1, 2}_{\nabla}$, 
which is degenerate since it does not contain increments in the direction $t$. The simplest case, just to understand the structure of the space would be to consider the space 
$\{u \in L^2\, :\,  \partial_{x_i} u \in L^2, i=1, \ldots, n-1\} .$
It is clear that we do not control the derivative in the last direction. To bypass this issue, we introduce a suitable Sobolev space taking into account the variable $t$.  We define

\begin{definition}
Assume $u \in \mathcal S(\mathbb R)$ the Schwartz space on $\R$. Then we define the operator $D_t^{1/2}$ by 
$$ 
 D^{\frac 12}_tu(t)=-\frac 1{2\sqrt {2\pi}}\int_{\mathbb R} \frac {u(t)-u(s)}{|t-s|^{\frac 32}}ds.$$
\end{definition}

\begin{remark}\label {homhalfderivative}
Formally the (half-)derivative $D_t^{1/2}$ has degree 1. The previous definition is the standard definition of the square-root of the second order derivative and can be found in the classical book of Landkof \cite{landkof}. It has also been instrumental (and furthermore investigated) in \cite{AEN,BGMN}. 
\end{remark}

This previous  definition can be extended to functions $v$ such that for
 every $\phi\in \mathcal S(\mathbb R)$ 
the product $vD^{\frac 12}_t\phi\in L^1(\mathbb R)$:
in this case we say that $D^{\frac 12}_t v= g \in L^2(\mathbb R)$  
  if 
$$\int  vD^{\frac 12}_t\phi dt=\int g\phi dt$$
 for
 every $\phi\in \mathcal S(\mathbb R)$. 
The definition coincides with the previous one in the special case that $v\in \mathcal S(\mathbb R)$ (see e.g. \cite {AEN,BGMN}). It is then classical to define the Sobolev space: let $\dot H^{\frac 12}(\mathbb R^{})$  
 the space which consists of all $v\in L^2_{loc (\mathbb R)}\cap \mathcal S'(\mathbb R)$  with norm 
$$||v||_{\dot H^{\frac 12}(\mathbb R^{})}= ||D^{\frac 12}_t v||_{L^2(\mathbb R^{})}.$$
For $v\in L^2{(\mathbb R)}$  we have  
$L^2(\mathbb R)\cap  \dot H^{\frac 12}(\mathbb R)=  H^{\frac 12}(\mathbb R)$   the usual fractional Sobolev space.

An important property of the derivative $D^{1/2}_t$ can be described by the Hilbert transform $H(t)$, the Calder\'on-Zygmund operator of symbol  $i\,\frac{\eta}{|\eta|}$. Then the following relation holds (see  \cite {AEN,BGMN})
$$\partial_t=D^{\frac 12}_t\, H_t\, D^{\frac 12}_t. $$

Note that,  
for any $u\in   \dot H^{\frac 12}(\mathbb R^{})$ and $\phi \in C^\infty_0 (\mathbb R^{})$
$$\int_{\mathbb R^{}} H_t\, D^{\frac 12}_tu\,    \overline{ D^{\frac 12}_t \phi}\, dt
=- \int_{\mathbb R^{}} u \,  \overline{ \partial_t \phi}\, dt.$$

\begin{remark}
Note that the half-order time derivative $D^{1/2}_t $ can be defined by the Fourier
multiplier $\tau \to |\tau|$, while the Hilbert transform $H(t)$ corresponds to the 
Fourier
multiplier $\tau \to i sgn(\tau)$. For this reason, we have
\begin{equation}\label{DHcomm}
D^{1/2}_t H(t) =  H(t)D^{1/2}_t.
\end{equation}
 For the same reason if $X_0$ commutes with $X_i$ - which means that 
we need to be in a parabolic (sub-Riemannian) setting, 
$H(t) $ commutes with any 
derivative $X_i$, $i=1, \ldots, m$: 
$$X_iH(t) =  H(t)X_i.$$
\end{remark}

Let us also recall that in a group structure there is a natural mollifier, defined as follows
\begin{definition}\label{moll}
Let $\psi$ be a smooth function such that 
$$\int \psi =1, \quad 0\leq \psi \leq1, $$
and for every function $f$ we will define 
$$f_\epsilon (x)  = \int f( \delta_\epsilon z \circ x) \psi (z) dz = \int f( y) \psi (\delta_{\epsilon^{-1}}|| x \circ y^{-1}||) \frac{dy}{\epsilon^Q}.$$

\end{definition}

Let us explicitly note that, due to the left invariance of the vector fields the following property holds
\begin{equation}\label{Xfe}
X_i f_\epsilon   = (X_i f)_\epsilon . 
\end{equation}

Note also that the operator of mollification is self adjoint, since 
\begin{equation}\label{feg}
\int f_\epsilon(x) g (x)  dx = \int g(x) \int f( y) \psi (\delta_{\epsilon^{-1}}|| x \circ y^{-1}||) \frac{dy}{\epsilon^Q}dx =\int f(x) g_\epsilon (x)  dx
\end{equation}
and the norm is symmetric.

\subsection{Fundamental solution and representation formulas }

The operator $\mathcal L$ defined in \eqref{op} is expressed in terms of a matrix $A$ uniformly positive. 
 Precisely there 
exists a constant $\nu$ such that 
\begin{equation} \label{A}
\nu|\xi|^2 \leq \sum_{i,j=1}^m a_{i,j}(x) \xi_i \xi_j \leq \nu^{-1 }|\xi|^2 
\end{equation}
for every $\xi\in \R^n$ and every $x\in \Omega$. Since the coefficients are only measurable, in general this operator will not have a fundamental solution, and we will use the fundamental solution  $\tilde \Gamma$ of the operator $\tilde{\mathcal L}$ introduced in \eqref{op_constant} instead. We list here some of the properties 
which we will need in the sequel

\begin{itemize}

\item[(i)] $\tilde \Gamma$ is a nonnegative function which is smooth away from the diagonal of $\R^n\times\R^n$;

\item[(ii)] $\tilde \Gamma(\cdot, y)$ and $\tilde \Gamma(x, \cdot)$ are locally integrable;

\item[(iii)] for every $u\in C^\infty_0(\R^n)$
$$u(\cdot)=-\int_{\R^n}\tilde \Gamma(\cdot, y)\,\tilde{\mathcal L}u(y)\,dy;
$$

\item[(iv)]  if $Q$ is the homogeneous dimension, the following estimates hold
\begin{equation}\label{estimGamma}
\tilde \Gamma(y, x)\leq c\,  d^{2-Q}(y, x), \quad |\nabla\tilde \Gamma (y, x)| \leq  c\, d^{1-Q}(y, x). 
\end{equation}
Moreover, by Remark \ref {homhalfderivative} we also have
$$ |D_t^{1/2}\tilde \Gamma(y, x)| \leq  c\, d^{1-Q}(y, x). $$
\end{itemize}

In the elliptic case the level sets of the fundamental solution are equivalent to the sphere of the metric. 
It is clear that this is not the case for parabolic or Kolmogorov type operators; however the 
level sets of  the fundamental solution play a crucial role for mean value formulas.

For this reason we introduce the following notations. For every $R>0$ and $x\in\R^n$,  let
$\Omega(x,R)$ denote a $R$-sublevel set of the fundamental solution $\tilde \Gamma$
$$\Omega(x, R):=\left \{ y\in \R^n | \ \tilde \Gamma(x ,y)>R^{2-Q}\right \}.$$

For every $R>0$ the set $\Omega(x,R)$ is a bounded nonempty set and for almost every $R>0$, the  set $\partial \Omega(x, R)$ is a $(n-1)$-dimensional $C^\infty$ manifold.

We will now introduce two definitions of average. We will denote the standard mean value on a set $U$ as follows 
\begin{equation}\label{defmean}
u_{U}:=\frac{1}{|U|}\int_{U}u(y)\, dy.
\end{equation} 

Solutions of the equation $\tilde {\mathcal L} u =f$ admit a natural representation in term of spatial mean value formula on the 
level sets of the fundamental solution. Precisely we define
\begin{equation}\label{defmean2}
 u_{\Gamma, \Omega(x,R)}:=\frac {Q}{(Q-2)
 R^Q}\int_{\Omega(x,R)} \frac{ |\nabla
 \tilde\Gamma(x,y)|^2}
{\tilde\Gamma^{2(Q-1)/(Q-2)}(x,y)}\,u (y)
\,dy
\end{equation}

\medskip

We recall the representation formula proved by Lanconelli and Pascucci in \cite{LP} (see Theorem 1.5). 

\begin{proposition}\label{repL}
The following holds
\begin{equation*}
 u(x)= u_{\Gamma, \Omega(x,R)}
-
\frac {Q}{R^Q}\int_0^{R} r^{Q-1}
\int_{\Omega(x,r)} \Big(\tilde\Gamma(x,y) - \frac{1}{r^{Q-2}}\Big) \tilde{ \mathcal L} u(y) \,dyd r. 
\end{equation*}
\end{proposition}

\section{Sobolev and Poincaré inequalities}

%
%
%
%
%
%
%

\subsection {Sobolev and Poincar\'e inequalities}

In this section we prove Sobolev and Poincar\'e type inequalities. 
Sobolev inequality will be obtained as a consequence of the representation of the solution in terms of the fundamental solution of the operator $\tilde {\mathcal L}$. The Poincar\'e inequality will be obtained as a direct consequence of the Sobolev inequality, using a compactness argument (see e.g. \cite{E}).
 
%

%
%
%
%

\begin{proposition} Let $U$ be a bounded set and assume that $u$ is a smooth function compactly supported in $U$, then
\begin{equation}\label{Sobolev1}
||u||_{L^{\frac{Q}{Q-1}}(U)} \leq C(diam(U) 
|| \mathcal L u||_{L^1 (U)}+||\nabla u||_{L^1(U)} ).
\end{equation}
\begin{equation}\label{Sobolev3}
||u||_{L^{\frac{Q}{Q-2}}(U)} \leq C( 
|| \mathcal L u||_{L^1 (U)}+||\nabla u||_{L^2(U)} ).
\end{equation}

Besides,  for every test function $\phi \in C^\infty_0$ we have 
\begin{equation}\label{notsure}
||(u^k\phi)^2||_{L^{\frac{Q}{Q-2}}(U)} \leq \end{equation}$$
\leq  Ck(  ||u^{2k-1}\phi^2 \mathcal L u||_{L^1 (U)}  +||\nabla (u^k \phi) ||_{L^2(U)}  + ||u^{2k} ( \phi^2 + |\nabla \phi|^2)||_{L^1(U)}).
$$
\end{proposition} 
\begin{proof} 

Since $u$ is compactly supported, we have

\begin{align*}u(x) &= \int \tilde \Gamma (x,y) \mathcal {\tilde  L}u(y) dy   \\
 &= \int \tilde \Gamma (x,y)   \mathcal Lu(y) dy   +
 \int \tilde \Gamma (x,y)  (\mathcal{\tilde L}- \mathcal L)u(y) dy\\
 &= \int \tilde \Gamma (x,y)  \mathcal Lu(y) dy  + 
\int  \left\langle (A- Id)
\nabla \tilde \Gamma(x,y), \nabla u(y)\right\rangle \,dy=I_1(x)+I_2(x),
\end{align*} 
where $A$ is the matrix of the coefficients introduced in \eqref{A} and $Id$ is the identity matrix.
By the estimates \eqref{estimGamma} on $\tilde \Gamma$ we immediately obtain 

$$||I_1||_{L^{\frac{Q}{Q-1}}(U)} \leq
diam(U) || \mathcal L u||_{L^1 (U)}
\text{ and }||I_2||_{L^{\frac{Q}{Q-1}}(U)} \leq
||\nabla u||_{L^1(U)}.$$
and $$||I_1||_{L^{\frac{Q}{Q-2}}(U)} \leq
|| \mathcal L u||_{L^1 (U)}
\text{ and }||I_2||_{L^{\frac{Q}{Q-2}}(U)} \leq
||\nabla u||_{L^2(U)}.$$

The proof of the last assertion is similar;
\begin{align*}(u^k \phi)^2(x) 
 &= \int \tilde \Gamma (x,y)  \mathcal L(u^k\phi)^2(y) dy   +
 \int \tilde \Gamma (x,y)  (\mathcal{\tilde L}- \mathcal L)(u^k \phi)^2(y) dy\end{align*} 
Note that 
$$\mathcal L(u^{2k} \phi^2) = 
-  X_i (a_{ij}( k X_j u  u^{2k-1}   \phi^2 +  2\phi X_j \phi  u^{k})   + kX_0u^{2k-1} \phi^2
+u^{2k} \phi X_0\phi = $$
$$=
- k X_i (a_{ij} X_j u)   u^{2k-1}   \phi^2 -  k a_{ij} X_j u  X_i u u^{2(k-1)}   \phi^2 - 2k a_{ij} X_j u   u^{2k-1} \phi   X_i\phi + $$$$ -  X_i (a_{ij}   2\phi X_j \phi  u^{k})    + kX_0u^{2k-1} \phi^2
+u^{2k} \phi X_0\phi = $$
$$= k \mathcal L u\,   u^{2k-1}   \phi^2 -  k a_{ij} X_j u  X_i u u^{2(k-1)}   \phi^2 - 2k a_{ij} X_j u   u^{2k-1} \phi   X_i\phi + $$$$ -  X_i (a_{ij}   2\phi X_j \phi  u^{2k})    
+u^{2k} \phi X_0\phi. $$
Then 
$$(u^k \phi)^2(x)  =  k \int \tilde \Gamma (x,y)   \mathcal L u   u^{2k-1}   \phi^2  dy  +\int \tilde \Gamma (x,y)   a_{ij} X_j (u^k \phi)  X_i (u^k \phi) dy  $$ 
$$ - \int \tilde \Gamma (x,y)   a_{ij} X_j \phi  X_i \phi u^{2k} dy - \int X_i  \tilde \Gamma (x,y)   a_{ij}   2\phi X_j \phi  u^{2k} dy$$ $$+\int  \left\langle (A- Id)
\nabla \tilde \Gamma(x,y),  u^k \phi \nabla (u^k \phi)(y)\right\rangle \,dy.
$$

\end{proof}

We have then 

\begin{lemma} Let $U$ be a bounded set and assume that $u$ is a smooth function supported in $U$. Then 
\begin{equation}\label{Sobolev1parab}
||u||_{L^{\frac{Q}{Q-2}}(U)} \leq C( 
||D^{1/2}_{t} u||_{L^2 (U)}+||\nabla u||_{L^2(U)} ). 
\end{equation}
\end{lemma}
\begin{proof}
To clarify notation, recall that a point will be denoted $x=(\hat x, t)$, where $\hat x$  is the spatial variable and $t$ the temporal one (see also Remark \ref{tnotation}) 

$$u(x) =- \int \tilde \Gamma (x, y) \mathcal{\tilde  L}u(y) dy  = $$
$$
\int  \left\langle 
\nabla \tilde \Gamma(x,y), \nabla u(y)\right\rangle \,dy - \int D^{1/2}_{t} \tilde \Gamma(x,y) H(t)D^{1/2}_{t}\nabla u(y)\,dy,$$
which implies  \eqref{Sobolev1parab}.

%
\end{proof}

We now state a representation formula in term of the operator $ \mathcal L$, suitably modifying the 
representation for the operator $\mathcal{\tilde L}$, recalled in Proposition \ref{repL}.

\medskip
\begin{proposition}  \label{reprsolid}
For every $x$ and $R>0$  such that
$ \overline {\Omega(x,R)}\subset \Omega$ and for every $u\in C^{\infty}(\Omega)$ we have
\begin{equation}\label{rappLu}
u(x)= u_{\Gamma, \Omega(x,R)}
+\frac {Q}{R^Q}\int_{0}^{R} r^{Q-1}
\int_{\Omega(x,r)} \left\langle
(\tilde A-A)\nabla\tilde\Gamma(x,y), \nabla u(y)\right\rangle \,dyd r
\end{equation}
$$
-
\frac {Q}{R^Q}\int_0^{R} r^{Q-1}
\int_{\Omega(x,r)} \Big(\tilde\Gamma(x,y) - \frac{1}{r^{Q-2}}\Big)\mathcal{ L} u(y) \,dyd r. 
$$
\end{proposition}

\begin{proof}

By Theorem 1.5 in \cite {LP} 
\begin{equation}
\label{}
 u(x)= u_{\Gamma, \Omega(x,R)}
-
\frac {Q}{R^Q}\int_0^{R} r^{Q-1}
\int_{\Omega(x,r)} \Big(\tilde\Gamma(x,y) - \frac{1}{r^{Q-2}}\Big) \mathcal {\tilde L} u(y) \,dyd r. 
\end{equation}
We have 
$$\mathcal{\tilde L}u=
\sum_{i,j=1}^m X_i (\delta_{i,j}-a_{i,j}) X_ju + \mathcal Lu.$$
Then, integrating by parts, one has
$$
 u(x)= u_{\Gamma, \Omega(x,R)}
+\frac {Q}{R^Q}\int_{0}^{R} r^{Q-1}
\int_{\Omega(x,r)} \left\langle
(\tilde A-A)\nabla\tilde\Gamma(x,y), \nabla u(y)\right\rangle \,dyd r- 
$$
$$
-
\frac {Q}{R^Q}\int_0^{R} r^{Q-1}
\int_{\Omega(x,r)} \Big(\tilde\Gamma(x,y) - \frac{1}{r^{Q-2}}\Big) \mathcal Lu(y) \,dyd r.
$$
\end{proof}

%
%
%
%

Let us now prove an {\sl ad hoc} version of the Poincar\'e inequality. 
Specifically we will see that the Poincaré inequality can be 
obtained as a simple consequence of the Sobolev inequality.

\begin{proposition} Let $U$ be a bounded connected, open set 
with $C^1$ boundary, and let $V\subset U$.
Then there exists a constant $C$ depending only on $Q, U, V$ such that  
$$|| u - u_V||_{L^1(U)} \leq C diam(U) \Big(diam(U) ||\mathcal Lu||_{L^1(U)} +||\nabla u||_{L^1(U)}\Big). $$
\end{proposition}

\begin{proof}

We first assume that the diameter of the set is equal to  $1$. We argue by contradiction. If the inequality is false, for ever integer $k$ there exists  a function $u_k$ satisfying 
$$||u_k-(u_k)_V||_{L^1(U)}> k\Big(||\nabla u_k||_{L^1(U)} +  || \mathcal L u_k||_{L^1(U)}\Big).$$
We normalize by defining 
\begin{equation}\label{umenomean}
v_k := \frac{u_k - (u_k)_V} {||u_k-(u_k)_V||_{L^1(U)}}.
\end{equation}
Then, by definition 
$$(v_k)_V =0; \quad ||v_k||_{L^1(U)} =1$$
and 
\begin{equation}\label{nablalaplace}
||\nabla v_k||_{L^2(U)} + ||\mathcal Lv_k||_{L^2(U)}<\frac{1}{k}.
\end{equation}
It follows that $v_k$ converges in $L^2(U)$ to a function $v$ such that 
$$||\nabla  v||_{L^1(U)} +  ||\mathcal Lv||_{L^1(U)}=0.$$
This means that $v$ is constant. On the other hand $(v_k)_V =0$, for every $k$ so that  $(v)_V =0$, 
which implies that $v=0$, but this is in contradiction with $||v_k||_{L^2(U)} =1$. 

The result for general open set $U$ follows by applying the present result to the function $u$ composed with the intrinsic dilations defined in \eqref{dilation}. 
\end{proof} 

The very same proof gives a Poincar\'e inequality in terms of parabolic means. 
\begin{proposition}\label{corpoincare}
If $U$ is a bounded connected set, $x_0 \in U$ and $R$ is such that 
$\Omega(x_0, R)\subset U$ then there exists a positive constant $C$ such that
$$||u-(u)_{\Gamma, \Omega(x_0, R)} ||_{L^1(U)}\leq C\, diam (U)\Big(||\nabla u||_{L^1(U)} +diam(U)  ||\mathcal Lu||_{L^1(U)}\Big)$$
\end{proposition}

As a corollary we obtain the pointwise estimate 


\begin{proposition}\label{pointPoinc}
 Let $u$ a regular function defined on an open, bounded connected set $U$. Let $x\in U$ and let $V\subset U$ be an open set containing  $\Omega(x, R)$. 
Then there exists a constant $C$ depending only on $Q, U,V$ such that 
$$|u(x) - u_V |\leq $$$$\leq 
  C\Big(|\mathcal V_{1/Q}( \nabla u+diam (V) \mathcal L u )(x)|+ diam^{-n+1}(V)  ||\nabla u  + diam(V) \mathcal L u||_{L^1(V)}\Big) $$
where $\mathcal V_{1/Q}$ is the Riesz potential defined by
\begin{equation}\label{V}
\mathcal{V}_\mu(u)(x) = \int   
d^{Q(\mu-1)}(x,y) | u(y)|dy. 
\end{equation}
\end{proposition}

\begin{proof}

$$| u(x) - u_V | \leq \Big| u(x) - \frac{1}{|V|}\int_V u(y) dy  \Big| =$$
$$=
\Big|u(x) - 
(u)_{\Gamma, \Omega(x, R)}\Big| +   \frac{1}{|V|} \int \Big|
(u)_{\Gamma, \Omega(x, R)}- 
u(y) \Big|  dy =I_1(x)+I_2(x). $$
In order to estimate $I_1$ we apply formula \eqref{rappLu} to get
\begin{equation*}
I_1(x) =  |u(x)-  u_{\Gamma, \Omega(x,R)}| \leq 
\frac {Q}{R^Q}\int_{0}^{R} r^{Q-1}
\int_{\Omega(x,r)} 
|(\tilde A-A)||\nabla\tilde\Gamma(x,y)||\nabla u(y)| \,dyd r+
\end{equation*}
$$
+
\frac {Q}{R^Q}\int_0^{R} r^{Q-1}
\int_{\Omega(x,r)} \Big|\tilde\Gamma(x,y) - \frac{1}{r^{Q-2}}\Big|\, | \mathcal L u(y)| \,dy\leq 
$$
$$\le 
\int_{\Omega(x,R)} 
|\tilde A-A|\,|\nabla\tilde\Gamma(x,y)|\,| \nabla u(y)|\,dy+ 
\int_{\Omega(x,R)} |\tilde\Gamma(x,y) | \, | \mathcal L u(y)| \,dy$$
by definition of the set $\Omega(x,R)$. Then we have 
$$\le 
\int_{\Omega(x,R)} 
|\tilde A-A|\,|\nabla\tilde\Gamma(x,y)|\,| \nabla u(y)|\,dy+ 
R \, \int_{\Omega(x,R)} |\tilde\Gamma(x,y) |^{\frac {Q-1}{Q-2}} \, | \mathcal L u(y)| \,dy$$
$$\le C\, |\mathcal V_{1/Q}( \nabla u+diam (V) \mathcal L u )(x)|
$$


By Poincar\'e inequality in Proposition \ref{corpoincare}, one has 
$$I_2(x)\le C\, diam^{-n+1}(V) ||\nabla u  + diam(V) \mathcal L u||_{L^1(V)}. $$


\end {proof}

In exactly  the same way, we can deduce from the Sobolev inequality the following pointwise Poincaré inequality for the heat operator. 
\begin{proposition} Assume that $X_0=\partial_t$. 
Let $u$ be a regular function defined on an open, bounded connected set $U$. Let $x\in U$ and let $V\subset U$ be an open set containing  $\Omega(x, R)$. 
Then there exists a constant $C$ depending only on $Q, U,V$ such that  
$$| u(x) - u_V |\leq  |\mathcal V_{1/Q}(\nabla u)(x) | + |\mathcal V_{1/Q }(D_t^{1/2}u)(x)| +$$
$$\qquad\qquad + diam^{-n +1}(V)\Big( ||\nabla u||_{L^2(V)} + ||D_t^{1/2}u||_{L^2(V)}\Big), 
$$
where $ \mathcal V_{1/Q}$ has been defined in \eqref{V}. 
\end{proposition}

\subsection{John-Nirenberg inequality}

Let us now prove the John-Nirenberg inequality, following the same approach as in Theorem 7.21 in Gilbarg-Trudinger \cite{GT}. We define 
\begin{definition} 
We say that $f \in M^p (\Omega) $ if 
$$\int_{\Omega\cap B_R} |f| dx \leq C R^{Q(1- 1/p)}.$$
We will also denote $||f||_{M^p(\Omega)}$ the infimum of the constant $C$ for which the previous inequality holds true. 

\end{definition} 
In particular one has $\Gamma\in  M^{Q/(Q-2)} (\Omega)$,  $\nabla\Gamma\in  M^{Q/(Q-1)} (\Omega)$, $ D^{1/2}_t\Gamma\in  M^{Q/(Q-1)} (\Omega)$. 

\begin{remark} 
As in the classical case (see \cite{GT} equation (8.4) and Theorem 8.15), we will assume that $u$ satisfies $\mathcal Lu =f$, with $f \in L^{q/2}$, and $q>Q$. This implies that 
$$R\int_{\Omega(x,R) } |\mathcal Lu| =R\int_{\Omega(x,R)}  |f(x)| dx \leq ||f||_{q/2} R^{1 + Q( 1 -\frac{2}{q})}  \leq ||f||_{q/2} R^{Q-1}. 
$$
\end{remark} 

\begin{lemma} \label{718}
Let $f\in M^p(\Omega)$. Then for every $\mu<1$ the Riesz potential 
$$\mathcal V_{\mu}(f)=\int_{\Omega} d^{Q(\mu -1)}(x,y) |f(y)| dy \leq \frac{p-1}{\mu p -1}
diam(\Omega)^{Q(\mu -1/p)}||f||_{M^p(\Omega)}.$$
\end{lemma} 
\begin{proof}
The lemma is inspired by Lemma 7.18 in \cite{GT}. 
Call 
$$v(R) = \int_{\Omega(x,R) } |f(y)| dy. $$
By coarea formula we have
$$v(R) = \int_{\Omega(x,R) } |f(y)| dy = \int_0^R \int_{\partial \Omega(x, \rho)}|f(s, \rho)|\frac{\Gamma^{(Q-1)/(Q-2)}}{|\nabla_E \Gamma|}dH_{Q-1}(s)d\rho,$$
hence
$$v'(\rho) = \int_{\partial \Omega(x, \rho)} |f|\frac{\Gamma^{(Q-1)/(Q-2)}}{|\nabla_E \Gamma|}dH_{Q-1}.$$
Then we have 
$$\mathcal V_{\mu}(f) =  \int_{\Omega} d^{Q(\mu -1)}(x,y) |f(y)| dy
= \int_0 ^{\text{diam}(\Omega)} \rho^{Q(\mu-1)}v'(\rho) d\rho=$$
integrating by parts 
$$= \text{diam}(\Omega)^{Q(\mu-1)}v(\text{diam}(\Omega)) - Q(\mu-1) \int_0 ^{\text{diam}(\Omega)} \rho^{Q(\mu-1)-1}v(\rho) d\rho\leq$$$$\leq
diam(\Omega)^{Q(\mu -1/p)}||f||_{M^p(\Omega)}.$$
\end{proof}

\begin{proposition} \label{JN}
Let $u$ be a smooth function defined on an open, bounded connected set $U_0$. Let $V\subset \subset U_0$, and let $R>0$
such that the set 
\begin{equation}\label{Verre}
V_{R}=\cup_{x \in V}  \Omega(x,R) \subset \subset  U_0
\end{equation}
If there is a constant K such that 
\begin{equation}\label{JN1}
\int_{V_R} |\nabla u| dx \leq K R^{Q-1} 
\end{equation}
and 
\begin{equation}\label{JN2}
R\int_{V_R} |\mathcal L u| dx \leq K R^{Q-1} \text{  or  }\int_{V_R}|D^{1/2}_t u| dx \leq K R^{Q-1} 
\end{equation}
Then 
$$\int_V \exp\Big( \frac{ | u(x) - u_V |}{K}\Big)\leq C 
 diam^Q (U_0).$$
\end{proposition} 
\begin{proof}
We prove the estimate  in case the first inequality in (\ref{JN2})  is verified.
We use Proposition \ref{pointPoinc}
to estimate $ | u(x) - u_V |$ 
$$|u(x) - u_V |\leq $$$$\leq 
  C\Big(|\mathcal V_{1/Q}( \nabla u+diam (V) \mathcal L u )(x)|+ diam^{1-n}(V)  ||\nabla u  + diam(V) \mathcal L u||_{L^1(V)}\Big). $$
In order to simplify notations we denote
$$f_u= \nabla u + diam(V)\mathcal Lu; \quad g_1(x) = |\mathcal V_{1/Q}(f_u)(x)| $$
Arguing exactly as in \cite{GT}, page 166 
and using the fact that $$-Q + 1 = Q\Big(\frac{1}{Qq} -1\Big) \frac{1}{q}+ Q\Big(\frac{1}{Q}+\frac{1}{Qq} -1\Big)\Big(1-\frac{1}{q} \Big),$$
we have 
$$g_1(x) \leq \Big(\int d^{Q(\frac{1}{Qq} -1) }(x,y) |f_u(y)| dy\Big)^{1/q} \Big(\int d^{ Q(\frac{1}{Q}+\frac{1}{Qq} -1) }(x,y) |f_u(y)| dy\Big)^{1-1/q}.
$$
Applying then Lemma \ref{718}
with $\mu=\frac{1}{Q}+\frac{1}{Qq}$, $p=Q$ to the second term, we deduce
$$
g_1(x) \leq  (q(Q-1)||f_u||_{M^p(\Omega)})^{(q-1)/q}diam(\Omega)^{(q-1)/q^2}\Big(\int d^{Q(\frac{1}{Qq} -1) }(x,y) |f_u(y)| dy\Big)^{1/q}.$$
Hence
$$
\int  |g_1(x)|^q dx \leq   q^{q-1}(Q-1)^{q-1} diam(\Omega)^{q-1/q}||f_
u||^{q-1}_{M^p(\Omega)}\int \int d^{(\frac{1}{q} -Q) }(x,y) |f_u(y)| dy\leq
$$
$$
 \leq  C_Q q^{q-1}(Q-1)^{q-1}  diam(\Omega)^{Q}||f_u||^{q}_{M^p(\Omega)}$$
$$= C_Q (q(Q-1) ||f_u||_{M^p(\Omega)})^{q}  diam(\Omega)^{Q}.$$
Clearly also 
$$
\int ||f_u||_{L^1}^q  \leq   C_Q (q(Q-1) ||f_u||_{L^1(\Omega)})^{q}  diam(\Omega)^{Q}\leq  $$$$\leq  C_Q (q(Q-1) ||f_u||_{M^p(\Omega)})^{q}  diam(\Omega)^{Q}.$$
So that 
$$
\int  |u(x)-u_V|^q dx \leq  C_Q (q(Q-1) ||f_u||_{M^p(\Omega)})^{q}  diam(\Omega)^{Q}.$$
Consequently 
$$
\int \sum_{q=0}^N \frac{  |u(x)-u_V|^q}{q! (c_1 K)^q} dx \leq 
C_Q \sum_{q=0}^N \Big(\frac{p-1}{c_1 }\Big)^q \frac{q^q}{q!} diam(\Omega)^{Q},$$
and letting $N\to +\infty$ we thus obtain the thesis.

\end{proof}

\section{Proof of the main results}

\subsection{Caccioppoli type inequality}

The classical parabolic Moser procedure, is an estimate of two terms: 
Caccioppoli and Sobolev inequality. In the previous section we presented a Sobolev inequality whose right hand side depends only on $\nabla u$ 
(and $\mathcal Lu$, which is in some $L^p$ by assumption).  As a consequence we need to express an estimate of Caccioppoli type, which estimates only the gradient, hence it is simpler than the classical parabolic one, and is comparable with the ones provided in \cite{AEN} for the heat  
or \cite{PP} for a special class of Kolmogorov equations.

We start with a technical lemma, which ensures that we can choose a power of the solution as a test function. We start with a truncation, in order to ensure integrability.

\begin{lemma} \label{serve?}Let $u$ be a weak solution of \eqref{eq}. 
Choose positive constants $M, M_0$, function $\phi\in C^\infty_0$, and call $ \tilde u_{M} = \min(u^+ , M) + M_0$ and  $ \tilde u = u^+  + M_0.$  
Then for all real numbers $k$ we have for $k  \not=0$
$$
\int  \sum_{i,j=1}^m a_{i,j}(x) X_j \tilde u X_i(\tilde u^{2k-1}_{ M} \phi^2) 
\leq $$
$$
 \int \frac{|\mathcal Lu|}{M_0}\tilde u^{2k}_{M} \phi^2+2\int   {\tilde u}^{2k}_{ M} \phi |X_0\phi| + 
2\int \tilde u\tilde u_{ M}^{2k-1}  \phi |X_0\phi|.
$$ 
For $k=0$ we have
\begin{equation*}\label{weak5}
\int  \sum_{i,j=1}^m a_{i,j}(x) X_j u X_i(\tilde u^{-1}_{ M} \phi^2) 
\end{equation*}
$$
\leq\int \frac{|\mathcal Lu|}{M_0} \phi^2+2\int |\log(  {\tilde u}_{M} )| \phi |X_0\phi |+ 
2\int \tilde u \tilde u_{M}^{-1}  \phi X_0\phi.
$$

\end{lemma}

\begin{proof}
Note that $\tilde u $ is a weak solution of $\mathcal L \tilde u = f \chi_{u>0}$, where $\chi$
is the characteristic function. Also note that, since $\phi$ is $C^{\infty}_0$, then 
$$\int  X_0 \phi =0$$
Hence the definition of weak solution given in \eqref{weak} is equivalent to 
$$
 \int (u +  M_0) X_0 \phi - \int  \sum_{i,j=1}^m a_{i,j}(x) X_j \tilde u X_i \phi =\int \mathcal L \tilde u \phi \quad \forall \phi \in C^\infty_0, 
$$
for every test function $\phi$. We set $ \tilde u_{\epsilon, M} = \min((u_\epsilon)^+ , M) + M_0$  and we choose as a test function the function $(\tilde u^{2k-1}_{\epsilon, M} \phi^2)_{\epsilon}$, where the index $\epsilon$ denote the mollification introduced  in Definition \ref{moll}. 
In this way we get 
\begin{equation*}
 \int (u +  M_0) X_0 ((\tilde u^{2k-1}_{\epsilon, M} \phi^2)_{\epsilon})- \int  \sum_{i,j=1}^m a_{i,j}(x) X_j \tilde u X_i((\tilde u^{2k-1}_{\epsilon, M} \phi^2)_{\epsilon})=\end{equation*}$$=\int \mathcal L\tilde u(\tilde u^{2k-1}_{\epsilon, M} \phi^2)_{\epsilon}\quad \forall \phi \in C^\infty_0, 
$$
Applying properties \eqref{Xfe} and \eqref{feg} 
\begin{equation}\label{weak1}
 \int (u_\epsilon  +  M_0) X_0 (\tilde u^{2k-1}_{\epsilon, M} \phi^2)- \int  \sum_{i,j=1}^m (a_{i,j}(x) X_j \tilde u)_\epsilon X_i(\tilde u^{2k-1}_{\epsilon, M} \phi^2)=\end{equation}$$=\int (\mathcal L\tilde u)_\epsilon \tilde u^{2k-1}_{\epsilon, M} \phi^2.
$$
Let us consider a term at a time
$$
 \int  (u_\epsilon  +  M_0)  X_0 (\tilde u^{2k-1}_{\epsilon, M} \phi^2) 
$$
$$
=(2k-1) \int_{0<u_\epsilon <M}\tilde u^{2k-1}_{\epsilon, M} X_0(  \tilde u_{\epsilon, M})\phi^2 +  
2\int_{u_\epsilon >0} (u_\epsilon  +  M_0) \tilde u_{\epsilon, M}^{2k-1}  \phi X_0\phi =
$$
$$
=- \frac{(2k-1)}{k} \int   {\tilde u}^{2k}_{\epsilon, M} \phi X_0\phi + 
2\int (u_\epsilon  +  M_0) \tilde u_{\epsilon, M}^{2k-1}  \phi X_0\phi,
$$
if $k\not=0$. 
Inserting in \eqref{weak1} and letting $\epsilon\to 0$ we get for $k\not=0$
\begin{equation}
\int  \sum_{i,j=1}^m a_{i,j}(x) X_j \tilde u X_i(\tilde u^{2k-1}_{ M} \phi^2) =
\end{equation} $$
-  \int \mathcal L \tilde u\tilde u^{2k-1}_{M} \phi^2- \frac{(2k-1)}{k} \int   {\tilde u}^{2k}_{ M} \phi X_0\phi + 
2\int \tilde u \tilde u_{ M}^{2k-1}  \phi X_0\phi\leq 
$$
We finally get
\begin{equation}\label{weak3}
\int  \sum_{i,j=1}^m a_{i,j}(x) X_j \tilde u X_i(\tilde u^{2k-1}_{ M} \phi^2) =
\end{equation} 
$$
\leq \int \frac{|\mathcal Lu|}{M_0}\tilde u^{2k}_{M} \phi^2+2\int   {\tilde u}^{2k}_{ M} \phi |X_0\phi| + 
2\int \tilde u\tilde u_{ M}^{2k-1}  \phi |X_0\phi|.
$$

Analogously for $k=0$ 
\begin{equation}\label{weak4}
\int  \sum_{i,j=1}^m a_{i,j}(x) X_j \tilde u X_i(\tilde u^{-1}_{ M} \phi^2) =
\end{equation} $$
\leq\int \frac{|\mathcal Lu|}{M_0} \phi^2+2\int |\log(  {\tilde u}_{M} )| \phi |X_0\phi |+ 
2\int \tilde u \tilde u_{M}^{-1}  \phi X_0\phi.
$$
 \end{proof}

Using the previous lemma we can now prove an elliptic-type Caccioppoli inequality for our operator. 

\begin{proposition} \label{iteratedCacciopoli}(iterated Caccioppoli type estimate)
Let $u$ be a weak solution of 
$\mathcal Lu = f$ in an open set $V$, and let $M_0$ be a positive constant .
If we call $ \tilde u = u^+  + M_0$, there exists a constant $C>0$ independent of $M_0$ such that

$$
\int |\nabla ( \log(\tilde u) \phi)|^2 
\leq C\int  \Big(\phi^2 +\frac{|\mathcal Lu|}{M_0} \phi^2 + \log(  \tilde u ) \phi |X_0\phi|  + |\nabla \phi|^2 \Big).
$$
If in addition $u\in L^{2k}$, then 
$$
\int |\nabla({\tilde  u}^k \phi)|^2 
\leq Ck\int  {\tilde  u}^{2k} \Big(\phi^2 +\frac{|\mathcal Lu|}{M_0} \phi^2+   \phi |X_0\phi| + |\nabla \phi|^2 \Big)
$$
for every $\phi \in C^\infty_0$.  The same properties are satisfied by the function $ - u^-  + M_0$.  
\end{proposition}

\begin{proof}

We note that, letting $M$ to $+\infty$ in Lemma  \eqref{serve?}
for $k  \not=0$ and $\tilde u \in L^{2k}$
\begin{equation}\label{weak2}
\int  \sum_{i,j=1}^m a_{i,j}(x) X_j \tilde u X_i(\tilde u^{2k-1} \phi^2) 
\leq \int\tilde u^{2k} \Big( \phi^2  \frac{|\mathcal Lu|}{M_0} +4 \phi |X_0\phi|\Big)
\end{equation}

For $k=0$ we have
\begin{equation}\label{weak5}
\int  \sum_{i,j=1}^m a_{i,j}(x) X_j \tilde u X_i(\tilde u^{2k-1}\phi^2) \leq
\int \Big(\frac{|\mathcal Lu|}{M_0} \phi^2 + 2 |\log(  {\tilde u})| \phi |X_0\phi |+ 
2   \phi |X_0\phi| + |\nabla \phi|^2 \Big),
\end{equation}

Let us now estimate the different terms in \eqref{weak2}. 
We have

$$ \sum_{i,j=1} ^m   \int a_{i,j}  X_i {\tilde  u}  X_j ( \tilde u_M^{2k-1} \phi^2) = $$

$$= \sum_{i,j=1} ^m   \int a_{i,j}  X_i {\tilde  u}  \Big((2k-1) X_j  \tilde u_M\tilde u_M^{2k-2} \phi^2 + 2\tilde u_M^{2k-1} \phi X_j \phi\Big) = 
$$

\begin{equation}\label{daqui}= \sum_{i,j=1} ^m   \int a_{i,j}  X_i {\tilde  u_M} \tilde u_M^{k-1} \phi \Big((2k-1) X_j  \tilde u_M\tilde u_M^{k-1} \phi + 2\tilde u_M^{k} \phi X_j \phi\Big) = 
\end{equation}

If $k\not=0$ we set $w= \frac{\tilde u_M^{k}}{k} $ and for  $k=0$ we set $w= \log(\tilde u_M)$. In both case $ X_i {\tilde  u_M} \tilde u_M^{k-1}  = X_i w$, so that  we  deduce 

$$ \sum_{i,j=1} ^m   \int a_{i,j}  X_i {\tilde  u_M}  X_j ( \tilde u_M^{2k-1} \phi^2) = $$

$$=  (2k-1)  \sum_{i,j=1} ^m   \int a_{i,j}  X_i w X_j w \phi^2 + 2 \sum_{i,j=1} ^m   \int a_{i,j}X_i w   \tilde u_M^{k} \phi X_j \phi. 
$$

%
%
%
%
%

From \eqref{weak2}  we obtain

$$
 (2k-1)  \sum_{i,j=1} ^m   \int a_{i,j}  X_i w X_j w \phi^2 + 2 \sum_{i,j=1} ^m   \int a_{i,j}X_i w   \tilde u_M^{k} \phi X_j \phi$$$$
\leq \int\tilde u^{2k} \Big( \phi^2  \frac{|\mathcal Lu|}{M_0} +2 \phi |X_0\phi| \Big)
$$

which, together with condition \eqref{A}  implies for $k\not=0$

$$ \int |\nabla(w\phi)|^2 \leq   \sum_{i,j=1} ^m   \int a_{i,j} X_i (w \phi)  X_j (w\phi)  \leq  \int\tilde u^{2k} \Big( \phi^2  \frac{|\mathcal Lu|}{M_0} +2 \phi |X_0\phi| + |\nabla \phi|^2\Big)
$$

For $k=0$ we obtain 
\begin{equation}\label{cacciop1}
\int |\nabla(w\phi)|^2 \leq C  \int \Big(\frac{|\mathcal Lu|}{M_0} \phi^2 + 2 |\log(  {\tilde u})| \phi |X_0\phi |+ 
2   \phi |X_0\phi| + |\nabla \phi|^2 \Big),
\end{equation}

\end{proof}

\subsection{Caccioppoli inequality for parabolic operators}

The limitation of the previous inequality is that it does not provide an estimate of $\partial_t u$. For this 
reason completely different estimates are usually proposed in the parabolic setting.
However in \cite{AEN} it was proposed to use $H(t) u \phi^2$ as a test function in the parabolic setting. 
The same choice can be made also in our setting: 

\begin{proposition}\label{prop42}Suppose  $u$  is a weak solution of $\mathcal Lu=f$ in an open set $V$
and let $M_0$ be a positive constant .
If we call $ \tilde u = u^+  + M_0$, there is a constant $C>0$ and 
independent of $M_0$ such that for every 
for $\phi\in C^{\infty}_0 ({\R^{n}})$ and for all real number $k\neq 0$

$$\int |D^{1/2}_t  ({\tilde u}^k\phi)  |^2  \leq $$$$
k\int f\phi \; H(t) ({\tilde u}^k \phi)  + \int {\tilde u}^k \partial_t \phi H(t) (u^k \phi)   + C  \int |\nabla ({\tilde u} ^k \phi)|^2 + 
\int {\tilde u} ^k |\nabla (\phi)|^2.$$
\end{proposition}
\begin{proof}

As before we use the fact that $\tilde u$ is a solution of $\mathcal Lu = f \chi _{u>0}$ 

Then, by the weak definition of the operator $\mathcal L$ we immediately have
$$\int \mathcal L \tilde u\,   H(t) ({\tilde u}^{2k-1} \phi^{2})=  \int \partial_t \tilde u \Big(H(t) ({\tilde u}^{2k-1} \phi^{2})\Big)+  \sum_{i,j=1}^m a_{i,j}(x) X_j \tilde u X_i \Big(H(t) ({\tilde u}^{2k-1} \phi^{2})\Big)=$$
(using the properties of $H$ described previously, and the fact that it commutes with $X_i$ for every $i$)
$$= \frac{1}{k}\int \partial_t ({\tilde u}^k \phi)H(t) ({\tilde u} ^{k}\phi)  -\frac{1}{k}\int {\tilde u}^k \partial_t \phi H(t) ({\tilde u}  ^{k} \phi) $$$$
+ (2k-1)\int  \sum_{i,j=1}^m a_{i,j}(x) X_j \tilde u H(t)  \Big(X_i \tilde u {\tilde u}^{2(k-1)} \phi^{2})\Big)
+ 2 \int  \sum_{i,j=1}^m a_{i,j}(x) X_j \tilde u H(t)  \Big( \phi {\tilde u}^{2k-1} X_i  \phi)\Big)
=$$
$$=\frac{1}{k} \int D^{1/2}_tH(t) D^{1/2}_t (u^k \phi) \:  H(t)( u^k \phi)   -\frac{1}{k}\int u^k \partial_t \phi H(t) (u^k \phi)=$$
$$
+ \frac{(2k-1)}{k^2}\int  \sum_{i,j=1}^m a_{i,j}(x) X_j ({\tilde u}^k) H(t)  \Big(X_i ({\tilde u}^{k}) \phi^{2})\Big)
+ \frac{2}{k} \int  \sum_{i,j=1}^m a_{i,j}(x) X_j (\tilde u)^k H(t)  \Big( \phi {\tilde u}^{k} X_i  \phi)\Big)
=$$
(integrating by parts)
$$=  \frac{1}{k} \int H(t) D^{1/2}_t ({\tilde u}^k \phi) \:  D^{1/2}_t H(t)( {\tilde u}^k \phi)   - \frac{1}{k} \int {\tilde u}^k \partial_t \phi H(t) ({\tilde u}^k \phi) $$
$$
+ \frac{(2k-1)}{k^2}\int  \sum_{i,j=1}^m a_{i,j}(x) X_j ({\tilde u}^k) H(t)  \Big(X_i ({\tilde u}^{k}) \phi^{2})\Big)
+ \frac{2}{k} \int  \sum_{i,j=1}^m a_{i,j}(x) X_j (\tilde u)^k H(t)  \Big( \phi {\tilde u}^{k} X_i  \phi)\Big)$$
(since $H(t)$ commutes with $D^{1/2}_t$ and $\nabla$)
$$= \frac{1}{k} \int |H(t) D^{1/2}_t ({\tilde u}^k\phi)   |^2  -\frac{1}{k} \int {\tilde u}^k \partial_t \phi H(t) (u^k \phi)   $$$$
+ \frac{(2k-1)}{k^2}\int  \sum_{i,j=1}^m a_{i,j}(x) X_j ({\tilde u}^k) H(t)  \Big(X_i ({\tilde u}^{k}) \phi^{2})\Big)
+ \frac{2}{k} \int  \sum_{i,j=1}^m a_{i,j}(x) X_j (\tilde u)^k H(t)  \Big( \phi {\tilde u}^{k} X_i  \phi)\Big)$$
If follows that 
$$\int |D^{1/2}_t  ({\tilde u}^k\phi)  |^2  \leq $$$$
k\int f\phi \; H(t) ({\tilde u}^k \phi)  + \int {\tilde u}^k \partial_t \phi H(t) (u^k \phi)   + C  \int |\nabla ({\tilde u} ^k \phi)|^2 + 
\int {\tilde u} ^k |\nabla (\phi)|^2$$
\end{proof}

\begin{corollary}Suppose  $u$  is a weak solution of $\mathcal Lu=f$. 
Let  $\phi\in C^{\infty}_0 (Q(\alpha r))$,  $\phi =1$ on the  parabolic cylinder $Q(r)$. 

If $u\in L^{2k}$, then we have 

\begin{align*} \label {Cacciop3}\int |\nabla(u\phi)|^2  d\xi + \int |H(t) D^{1/2}_t  (u\phi)  |^2
\le & 
Ck\int  {\tilde  u}^{2k} \Big(\phi^2 +\frac{|\mathcal Lu|}{M_0} \phi^2+   \phi |X_0\phi| + |\nabla \phi|^2 \Big).
\end {align*}

In addition 
$$
\int |\nabla ( \log(\tilde u) \phi)|^2 
\leq C\int  \Big(\phi^2 +\frac{| \mathcal Lu|}{M_0} \phi^2 + \log(  \tilde u ) \phi |X_0\phi|  + |\nabla \phi|^2 \Big).
$$

\end{corollary}

\subsection{A contrario Young inequality}
From the Caccioppoli and Sobolev inequalities we deduce in a standard way an {\sl a contrario} inequality. 

\begin{proposition}\label{CY}

Let $u$ satisfy weakly 
$ \mathcal Lu = f\in L^{q/2}$, with  $Q < q $ in an open set $U$. 
We choose $M_0\geq ||f||_{L^q}$ and $ \tilde u = u^+  + M_0$, then we have

$$\Big(\int_V |\tilde u^{k} \phi|^{2Q/(Q-2)}\Big)^{(Q-2)/Q} \leq  
C \int \tilde u^{2k } (\phi^2 + |\nabla \phi|^2 + \phi |X_0\phi|). 
$$
\end{proposition}

\begin{proof} We have by \eqref{notsure}
$$\Big(\int_V |\tilde u^{k} \phi|^{2Q/(Q-2)}\Big)^{(Q-2)/Q} \leq $$

$$
\leq \int |f|  \tilde u ^{2k-1 } \phi^2 + 
\int \tilde u^{2k } (\phi^2 + |\nabla \phi|^2)  +  2k^2\int  |\nabla \tilde u|^2 \tilde u^{2k-2} \phi^2 
$$
(by the iterated Cacciopoli)
$$
\leq \int \frac{|f|}{M_0} \tilde u^{2k} \phi^2 + 
\int \tilde u^{2k } (\phi^2 + |\nabla \phi|^2 + \phi |X_0\phi|). 
$$
Note that
$$
\int \frac{|f|}{M_0} \tilde u^{2k} \phi^2 
\leq \Big(\int (\frac{|f|}{M_0})^{q/2}\Big)^{2/q}\Big(\int (\tilde u^{k} \phi)^{2q/(q-2)} \Big)^{(q-2)/q}\leq 
$$
$$
\leq  \sigma\Big(\int (\tilde u^{k} \phi)^{2Q/(Q-2)}  \Big)^{(Q-2)/Q}+ \frac{1}{\sigma} \int \tilde u^{2k} \phi^2.
$$
Then, for an appropriate choice of $\sigma$, we have the thesis.
\end{proof}

\subsection{Moser iteration and $C^{\alpha}$ regularity}

We now apply to these general Kolmogorov operators, the Moser iteration technique introduced for the elliptic operators and not to the parabolic ones, which is more technical. We get 

%
%

\begin{proposition}\label{prosup}
Assume that $u$ is a weak solution of $\mathcal Lu=f$  with   $f\in L^q$, $q> Q/2$   in an open set $U_0$.
Call $M_0 $ a constant such that $M_0\geq ||f||_{L^q(B(x, 2R)),}$  and $\tilde u = u^+ + M_0. $
Let $B (x, R)$  be a ball, such that $B(x, 2R)\subset \subset U_0$    and let  $p>1$.  Then there exists a constant $C$ only depending on $Q, \nu,  q $    and p  such that 

$$
\sup_{B(x, R)} \tilde u \leq  C  \Big( \frac{1}{B(x, 2R)} \int_{B(x, 2R)} \tilde u^p \Big) ^{1/p},
$$
$$
\inf_{B(x, R)}  \tilde u \geq  C  \Big( \frac{1}{B(x, 2R)} \int_{B(x, 2R)} \tilde u^{-p} \Big) ^{ -1/p},
$$

The same assertion holds also for $\tilde u = -u^- + M_0. $

\end{proposition}

\begin{proof}
We can apply the standard proof, known in the elliptic setting, as can be found in \cite{GT}. 
\end{proof}
%
%
%
%
%

\begin{proposition}
Assume that $u$ is a solution of $\mathcal Lu=f\in L^q$, with  $Q/2 < q  $ in an open set $U_0$. 
Let $x_0\in U_0$, and assume that $\cup_{x \in B(x, 2R)}  \Omega(x,R) \subset \subset  U_0 $. Call $M_0 = \max ( ||f||_{L^q}, 2 \sup |u|)$  and $\tilde u = u + M_0. $
Then there exists  $p$ and a positive constant $C$ depending on $\sup u$, and $||f||_{L^q}$ such that 

$$
\Big( \int_{B(x, R)} {\tilde u}^{p} \Big) \Big( \int_{B(x, R)} {\tilde u}^{-p} \Big) \leq C. 
$$
\end{proposition}

\begin{remark}
Notice the loss of uniform invariance in the previous proposition.
\end{remark}

\begin{proof}
Let us verify the assumptions of Proposition \ref{JN}. 
Let us call $w = log (\tilde u)$ Let us compute 
$Lw$: 
%
$$\mathcal L w = \frac{1}{\tilde u} \mathcal L\tilde u - a_{ij}X_i w X_j w$$
Then 
$$R \int | \mathcal L\tilde u| \phi^2 \leq \Big(\int  f^q \phi^2\Big)^{1/q} R^{(1-1/q)Q+1}\leq R^{Q-1}$$
(by Proposition \ref{iteratedCacciopoli})
$$
\int |\nabla ({w} \phi)|^2 \leq 
 R\int  \big(\phi^2 +\frac{|\mathcal Lu|}{M_0} \phi^2 + |\log(  \tilde u )| \phi |X_0\phi|  + |\nabla \phi|^2 \big)\leq $$
since $Q/q\leq2.$ In addition 
$$\int_{B_R} |\nabla ({w} \phi)|\leq R^{Q/2}\Big(\int |\nabla ({w} \phi)|^2 \Big)^{1/2}
$$
(by Proposition \ref{iteratedCacciopoli})
$$
\leq R^{Q/2}C\Big(\int  \big(\phi^2 +\frac{|\mathcal Lu|}{M_0} \phi^2 + |\log(  \tilde u )| \phi |X_0\phi|  + |\nabla \phi|^2 \big)\Big)^{1/2}
$$
$$\leq  R^{Q-1} + \Big(\int  f^q \phi^2 \Big)^{1/q} R^{Q(1- 1/q)/2 + Q/2} \leq R^{Q-1},$$
since $q\geq Q/2$.  

Now the conclusion follows as in Gilbarg Trudinger \cite{GT}, page 198. 
\end{proof}

{\sl Acknowledgements}:  YS is partially funded by  DMS Grant $2154219$, " Regularity {\sl vs} singularity formation in elliptic and parabolic equations". The work started in the occasion of a visit of GC at Johns Hopkins University. She would like to thank the Department for the hospitality.   Part of this work has been carried out at the occasion of a visit of YS at Universita di Bologna. He would like to thank the Department for the hospitality. Part of the work has been carried out while YS and GC were in residence at Institut Mittag-Leffler in Djursholm, Sweden during the semester "Geometric Aspects of Nonlinear Partial Differential Equations", supported by the Swedish Research Council under grant no. 2016-06596. GC is partially funded by the  project RISE EU H2020 GHAIA Grant n. 777822.

\bibliographystyle{abbrv} 
\bibliography{bib}

\end{document}